\documentclass[11pt,a4paper]{article}
\usepackage[utf8]{inputenc}
\usepackage[english]{babel}
\usepackage{amsmath}
\usepackage{amsfonts}
\usepackage{amssymb}
\usepackage{graphicx}
\usepackage{xcolor}
\usepackage{amsthm}
\usepackage{hyperref}
\usepackage[left=2cm,right=2cm,top=2cm,bottom=2cm]{geometry}

\title{Uniform in time modulus of continuity of Brownian motion}
\author{Julien Chevallier
\footnote{Corresponding author: {\sf{e-mail: \href{mailto:julien.chevallier1@univ-grenoble-alpes.fr}{julien.chevallier1@univ-grenoble-alpes.fr}}}} \medskip \\
Univ. Grenoble Alpes, CNRS, Inria, Grenoble INP*, LJK, 38000 Grenoble, France\\
* Institute of Engineering Univ. Grenoble Alpes}

\newtheorem{theorem}{Theorem}
\newtheorem{lemma}{Lemma}
\newtheorem{definition}{Definition}
\newtheorem{proposition}{Proposition}
\newtheorem{assumption}{Assumption}
\newtheorem{corollary}{Corollary}
\newtheorem{remark}{Remark}

\begin{document}
\maketitle

\begin{abstract}
    Let $B=(B_t)_{t\geq 0}$ be a standard Brownian motion. The main objective is to find a uniform (in time) control of the modulus of continuity of $B$ in the spirit of what appears in \cite{kurtz1978strong}. More precisely, it involves the control of the exponential moments of the random variable $\sup_{0\leq s\leq t} |B_t-B_s|/w(t,|t-s|)$ for a suitable function $w$. A stability inequality for diffusion processes is then derived and applied to two simple frameworks.

    \vspace*{-1em}
\end{abstract}

\section{Introduction}

Let $f:\mathbb{R}_+\to \mathbb{R}$ be some function and $T>0$ be a positive time horizon. Then, the modulus of continuity of $f$ on $[0,T]$ is the function $\omega_{f}(T,\cdot)$ defined by, for all $h\leq T$,
\begin{equation*}
    \omega_{f}(T,h) := \sup_{0\leq s<t\leq T, |t-s|\leq h} |f(t)-f(s)|.
\end{equation*}
Let $B = (B_t)_{t\in \mathbb{R}_+}$ be a standard Brownian motion living on the probability space $(\Omega,\mathcal{F},\mathbb{P})$ and denote by $\omega_{B}$ its pathwise modulus of continuity as defined above. Of course, this function depends on the path of $B$ and in turn is random. Perhaps the most known result about $\omega_{B}$ is Lévy's modulus of continuity theorem \cite[p. 172]{levy1937theorie}, which gives the following equivalent of $\omega_{B}$ for small $h$,
\begin{equation*}
    \omega _{B}(1, h) \sim_{h\to 0} {\sqrt {2h \ln \frac{1}{h}}} \quad \text{almost surely.}
\end{equation*}
More recently, some bounds were obtained on $\omega_{B}$. 

On the one hand, \cite{fischer2009the} proves that for all $p>0$, there exists an explicit constant $C(p)$ such that for all $T>0$ and $h\leq T$,
\begin{equation*}
    \mathbb{E}\left[ \left( \omega_{B}(T,h) \right)^{p} \right] \leq C(p) \left( \sqrt{h\ln\left( \frac{2T}{h} \right)} \right)^{p}.
\end{equation*}
Moreover, this bound is also derived for general Itô processes. Those results are then applied to the control of the Euler approximation of stochastic delay differential equations.

On the other hand, the Remark below Lemma 3.2. in \cite{kurtz1978strong} states that the random variable 
\begin{equation*}
    M = \sup_{\substack{0< s<t< T\\ h=|t-s|}} \frac{|B_{t}-B_{s}|}{\sqrt{h \left(1+ \ln\frac{T}{h} \right)}},
\end{equation*}
is such that $M^2$ admits exponential moments, that is $\mathbb{E}\left[ \exp(\lambda M^{2}) \right] < \infty$ for some $\lambda>0$. Of course, it is related with the following bound for the modulus of continuity, 
\begin{equation}\label{eq:control:Kurtz}
    \omega_{B}(T,h) \leq M \sqrt{h \left(1+ \ln\frac{T}{h} \right)}.
\end{equation}
Those results are applied to the derivation of strong diffusion approximation of jump processes. Even if it is not clear from the notation, the random variable $M$ depends on $T$ so that the equation above does not give a bound for the uniform (with respect to $T$) modulus of continuity.

The main objective of this paper is to prove the following bound for the uniform modulus of continuity of Brownian motion.
\begin{theorem}\label{thm:expo:moments}
    Let $B$ be a standard Brownian motion. Let $\varepsilon>0$ and define the random variable
    \begin{equation*}
    M_B := \sup_{\substack{0< s<t<\infty, \\ h=|t-s|}} \frac{|B_{t}-B_{s}|}{\sqrt{h \left(1+ \ln\frac{t}{h} + \varepsilon |\ln t| \right)}}.
    \end{equation*}
    Then, $M_B^{2}$ admits exponential moments.
\end{theorem}
The second objective is to derive a stability inequality for diffusion processes in the spirit of what appear in the proofs of the strong diffusion approximation in \cite{kurtz1978strong}. Finally, two applications of this inequality are given.

The paper is organized as follows. Some properties of the quantity that appears in the denominator in the definition of $M_{B}$ are stated in Section \ref{sec:about:w}. A stability inequality for diffusion processes is proved in Section \ref{sec:stability}. This inequality can be used to prove convergence results in the framework of small perturbations of the coefficients of the diffusion in two different frameworks (see Section \ref{sec:applications}). Finally, the proof of the main result is given in Section \ref{sec:proof:thm}.

\section{About the upper-bound}
\label{sec:about:w}
Let $0<\varepsilon<1$ in this section, and define $w:\mathbb{R}_+^*\times \mathbb{R}_+^*\to \mathbb{R}$ by, for all $0<h\leq t $,
\begin{equation}
    w(t,h) := \sqrt{h \left(1+\ln \frac{t}{h} + \varepsilon |\ln t| \right)}, \quad \text{ and } \quad w(t,h) := w(t,t) \text{ if $0<t<h$.}
\end{equation}

The value $w(t,h)$ for $h\leq t$ is linked with Theorem \ref{thm:expo:moments}, whereas it is defined for $h> t$ in order to satisfy some monotony (see Proposition \ref{prop:inequality:omega}). Notice that, in comparison with Equation ??, it is also natural to consider the function $w_K:\mathbb{R}_+^*\times \mathbb{R}_+^*\to \mathbb{R}$ defined by, for all $0<h\leq t $,
\begin{equation*}
    w_K(t,h) := \sqrt{h \left(1+\ln \frac{t}{h} \right)}, \quad \text{ and } \quad w_K(t,h) := w_K(t,t) \text{ if $0<t<h$.}
\end{equation*}
Of course, $w_K(t,h) \leq w(t,h)$ but $w_K$ controls the the modulus of continuity for finite time horizons whereas $w$ gives a uniform control. For instance, one can compare Equation \eqref{eq:control:Kurtz} with the following corollary of Theorem \ref{thm:expo:moments}.

\begin{corollary}\label{cor:control:modulus}
    There exists a random variable $M_B$ such that $M_B^{2}$ has exponential moments and for all $0<h<t<+\infty$,
    \begin{equation*}
    \omega_{B}(t,h) \leq M_B\, w(t,h).
    \end{equation*}
\end{corollary}

\begin{remark}
    Theorem \ref{thm:expo:moments} and in turn Corollary \ref{cor:control:modulus} are not valid with $w$ replaced by $w_K$. In that case, we would have $M_B = \infty$ almost surely. Indeed, by considering $s \to 0$, $M_B$ would be larger than $\sup_{0< t<\infty} t^{-1/2}|B_{t}|$. Yet, the law of the iterated logarithm yields that $t^{-1/2}|B_{t}| \sim \sqrt{2 \ln \ln t} \to \infty$ a.s. as $t\to \infty$.
\end{remark}

Here are listed two nice properties satisfied by our upper-bound function $w$.

\begin{proposition}\label{prop:inequality:omega}
The function $w$ is non decreasing, that is
\begin{equation*}
    \forall t'\geq t\geq 0,\, h'\geq h\geq 0,\quad w(t',h')\geq w(t,h).
\end{equation*}
\end{proposition}
\begin{proof}
    Let $t'\geq t\geq 0$ and $h'\geq h\geq 0$. It is clear that $w(t',h) \geq w(t,h)$ and it only remains to prove that $w(t,h')\geq w(t,h)$. 
    
    The function $h\mapsto h(1+\ln(a/h))$ is non-decreasing for all positive $h\leq a$ which directly implies (with $a=\max\{t^{p},t^{q}\}$) that $w(t,h)\leq w(t,\min\{t,h'\})= w(t,h')$ by definition.
\end{proof}

Obviously, the function $w_K$ is also non decreasing. Furthermore, the function $w_K$ satisfies the same scaling invariance as the Brownian motion. More precisely, for all $a>0$, $w_K(at,ah) = \sqrt{a} w_K(t,h)$. This property is almost satisfied by the function $w$ in the following sense.

\begin{proposition}\label{prop:scaling:omega}
    For all $a>0$, $(t,h)\in (\mathbb{R}_+^*)^2$, 
    \begin{equation}
        \frac{1}{1+\sqrt{\varepsilon |\ln a|}}\leq \frac{w(at,ah)}{\sqrt{a} w(t,h)} \leq 1+\sqrt{\varepsilon |\ln a|}.
    \end{equation}
\end{proposition}
\begin{proof}
    Let $a>0$ and $(t,h)\in (\mathbb{R}_+^*)^2$. Assume that $t\geq h$. We have
    \begin{equation*}
        w(at,ah) = \sqrt{ah \left( 1+\ln \frac{at}{ah} + \varepsilon |\ln at| \right)} \leq  \sqrt{ah\left( 1 + \ln \frac{t}{h} + \varepsilon |\ln t| + \varepsilon | \ln a | \right)}.
    \end{equation*}
    Using the fact that $\sqrt{b+c} \leq \sqrt{b}+\sqrt{c}$ when $b,c\geq 0$ and the fact that $\sqrt{h}\leq w(t,h)$, we get $w(at,ah) \leq \sqrt{a} w(t,h) (1+ \sqrt{\varepsilon |\ln a |})$ which corresponds to the upper bound and the same kind of argument gives the lower bound.

    Finally, if $h>t$, then the same kind of argument can be applied to $w(at,ah) = w(at,at)$ and $w(t,h)=w(t,t)$.
\end{proof}

In particular, this property can be used to compare the modulus of continuity of the Brownian motion $B$ and its space-time scaling. More precisely, we have the following corollary of Theorem \ref{thm:expo:moments}.

\begin{corollary}\label{cor:scaled:BM}
    Let $a>0$ and define the scaled Brownian motion $\tilde{B}$ by $\tilde{B}_t = a^{-1/2} B_{at}$ for all $t\geq 0$. Then, 
    \begin{equation*}
        M_{\tilde{B}} = \sup_{0\leq s<t<+\infty} \frac{|\tilde{B}_{t}-\tilde{B}_{s}|}{w(t,|t-s|)} \leq \left( 1 + \sqrt{\varepsilon |\ln a|} \right) M_{B},
    \end{equation*}
    where $M_{B}$ is the random variable defined in Theorem \ref{thm:expo:moments}.
\end{corollary}
\begin{proof}
    Let $a>0$ and $0\leq s<t<+\infty$. By Proposition \ref{prop:inequality:omega}, we know that $\sqrt{a}w(t,|t-s|) \geq w(at,a|t-s|)/(1+\sqrt{\varepsilon |\ln a|})$. Hence, 
    \begin{equation*}
        \frac{|\tilde{B}_{t}-\tilde{B}_{s}|}{w(t,|t-s|)} = \frac{|B_{at} - B_{as}|}{\sqrt{a} w(t,|t-s|)} \leq \frac{|B_{at} - B_{as}|}{w(at,a|t-s|)} \left(  1+\sqrt{\varepsilon |\ln a|}\right),
    \end{equation*}
    which gives the result.
\end{proof}

\section{Stability inequality for diffusions}
\label{sec:stability}

The main result of this section is Proposition \ref{prop:stability:inequality:diffusion}. It is a general stability inequality for diffusions which can be used in particular to provide explicit rates for strong convergence results (see Section \ref{sec:applications}).

\subsection{Setting}

Let $X$ and $\overline{X}$ be two diffusion processes satisfying
\begin{equation}\label{eq:EDS:diffusion}
\begin{cases}
X(t) = x_{0} + \int_{0}^{t} b(s,X(s)) ds + B(\Lambda(t))\\
\overline{X}(t) = \overline{x}_{0} + \int_{0}^{t} \overline{b}(s,\overline{X}(s)) ds + B(\overline{\Lambda}(t)),
\end{cases}
\end{equation}
where 
\begin{equation*}
\Lambda(t) := \int_{0}^{t} \sigma(s,X(s))^{2}ds \quad \text{ and } \quad \overline{\Lambda}(t) := \int_{0}^{t} \overline{\sigma}(s,\overline{X}(s))^{2}ds.
\end{equation*}

In the whole paper, we assume that those two equations admit strong solutions. For instance, this is guaranteed if the drift $b$ and the diffusion $\sigma$ are both sub-linear and Lipschitz functions (see \cite[§5.2]{oksendal2003stochastic} for instance). For instance, the equation for $X$ is equivalent to the Ito equation
\begin{equation*}
    X(t) = x_0 + \int_{0}^{t} b(s,X(s)) ds + \int_{0}^{t} \sigma(s,X(s)) dW_s,
\end{equation*}
where $W$ is a standard Brownian motion. See for instance \cite{kurtz1976limit} for details on this equivalence.

\begin{definition}
The functions $g,\overline{g} : \mathbb{R}_{+}\times \mathbb{R}^{d} \to \mathbb{R}^{k}$ are said to be \emph{Lipschitz-bounded-close} with constant $L$ and non-decreasing functions $K,D:\mathbb{R}_{+}\to \mathbb{R}_{+}$, abbreviated as LBC$(L,K,D)$, if for all $t\in \mathbb{R}_{+}$ and $x,\overline{x}\in \mathbb{R}^{d}$,
\begin{equation}\label{eq:def:LBC}
\begin{cases}
|\overline{g}(t,x) - \overline{g}(t,\overline{x})| \leq L |x-\overline{x}|,\\
\max\left\{ |g(t,x)|, |\overline{g}(t,x)| \right\} \leq K(|x|),\\
|g(t,x) - \overline{g}(t,x)| \leq D(|x|).
\end{cases}
\end{equation}
By extension, $g: \mathbb{R}_{+}\times \mathbb{R}^{d} \to \mathbb{R}^{k}$ is said to be \emph{Lipschitz-bounded} with constant $L$ and non-decreasing function $K$, abbreviated as LB$(L,K)$, if the first two lines of \eqref{eq:def:LBC} are satisfied.
\end{definition}
Here are gathered the assumptions made on the parameters of the model.

\begin{assumption}\label{ass:Lipschitz:close}
The functions $b,\overline{b}:\mathbb{R}_{+}\times \mathbb{R}\to \mathbb{R}$ are LBC$(L_{b},K_{b},D_{b})$ and the functions $\sigma,\overline{\sigma}:\mathbb{R}_{+}\times \mathbb{R}\to \mathbb{R}$ are such that $\sigma^{2}$ and $\overline{\sigma}^{2}$ are LBC$(L_{\sigma},K_{\sigma},D_{\sigma})$.
\end{assumption}

\subsection{The result}

In this following, we denote $X^*(t) = \sup_{0\leq s\leq t} |X(s)|$ and $\overline{X}^*(t) = \sup_{0\leq s\leq t} |\overline{X}(s)|$. The following result is highly related to and inspired from \cite[Lemma 3.2.]{kurtz1978strong}.

\begin{proposition}\label{prop:stability:inequality:diffusion}
Let $X$ and $\overline{X}$ be the two processes defined by \eqref{eq:EDS:diffusion}. Let $M_B$ be the random variable defined in Theorem \ref{thm:expo:moments}.

Under Assumption \ref{ass:Lipschitz:close}, for all $T>0$, $\gamma(T):= \sup_{0\leq t\leq T} |X(t) - \overline{X}(t)|$ satisfies
\begin{multline}\label{eq:main:control:diffusion}
\gamma(T) \leq 1 + 2e^{2L_{b}T} \Big[ |x_{0}-\overline{x}_{0}| + TD_{b}(X^*(t)) + M_B\, w(TK_{\sigma}(X^*(t)), TD_{\sigma}(X^*(t))) \\
+ M_B^{2}\, w\left( TK_{\sigma}(X^*(t)+\overline{X}^*(t)) , TL_{\sigma} \right)^{2} \Big].
\end{multline}
\end{proposition}
\begin{proof}
Let us define the intermediate integrated diffusion coefficient $\tilde{\Lambda}$ as
\begin{equation*}
\tilde{\Lambda}(t) = \int_{0}^{t} \overline{\sigma}^{2}(s,X(s)) ds.
\end{equation*}
The assumptions made on $\overline{\sigma}^{2}$ imply that
\begin{equation*}
\Lambda(t) \leq K_{\sigma}(X^*(t)), \quad \tilde{\Lambda}_{i}(t) \leq K_{\sigma}(X^*(t)) \quad \text{and} \quad \overline{\Lambda}_{i}(t) \leq K_{\sigma}(\overline{X}^*(t)).
\end{equation*}

The difference between $X$ and $\overline{X}$ can be decomposed into $X(t)-\overline{X}(t) = \sum_{j=1}^{5} A_{j}(t)$ with
\begin{equation*}
\begin{cases}
A_{1}(t) := x_{0}-\overline{x}_{0},\\
A_{2}(t) := \int_{0}^{t} b(s,X(s)) - \overline{b}(s,X(s)) ds,\\
A_{3}(t) := \int_{0}^{t} \overline{b}(s,X(s)) - \overline{b}(s,\overline{X}(s)) ds,\\
A_{4}(t) := B(\Lambda(t)) - B(\tilde{\Lambda}(t)),\\
A_{5}(t) := B(\tilde{\Lambda}(t)) - B(\overline{\Lambda}(t)).\\
\end{cases}
\end{equation*}
Thanks to the assumptions on the model, we have
\begin{equation*}
\begin{cases}
|A_{1}(t)| \leq |x_{0}-\overline{x}_{0}|,\\
|A_{2}(t)| \leq tD_{b}(X^*(t)),\\
|A_{3}(t)| \leq L_{b} \int_{0}^{t} |X(s) - \overline{X}(s)| ds.
\end{cases}
\end{equation*}
The last two terms, namely $A_{4}(t)$ and $A_5(t)$, can be bounded by using the monotony of $w$. On the one hand, since $\max\{\Lambda(t), \tilde{\Lambda}(t)\} \leq tK_{\sigma}(X^*(t))$ and $|\Lambda(t) - \tilde{\Lambda}(t)| \leq tD_{\sigma}(X^*(t))$ we have
\begin{equation*}
|A_{4}(t)| \leq M_B \, w(tK_{\sigma}(X^*(t)), tD_{\sigma}(X^*(t))).
\end{equation*}
On the other hand, remind that $\gamma(T)=\sup_{0\leq t\leq T} |X(t) - \overline{X}(t)|$. Since $\max\{\tilde{\Lambda}(t), \overline{\Lambda}(t)\} \leq tK_{\sigma}(X^*(t) + \overline{X}^*(t))$ and, using the fact that $\overline{\sigma}^2$ is Lipschitz, $|\tilde{\Lambda}(t) - \overline{\Lambda}(t)| \leq L_{\sigma} \int_{0}^{t} |X(s) - \overline{X}(s)| ds \leq TL_{\sigma}\gamma(T)$, we have 
\begin{equation*}
    |A_{5}(t)| \leq M_B \, w\left( T K_{\sigma}(X^*(T) + \overline{X}^*(T)) , TL_{\sigma}\gamma(T) \right).
\end{equation*}

Hence, Gronwall's Lemma gives, for all $t\leq T$,
\begin{equation}\label{eq:Gronwall:step}
|X(t) - \overline{X}(t)| \leq e^{L_{b} t} \left(\Delta (T) + M_B \, w\left( T K_{\sigma}(X^*(T) + \overline{X}^*(T)) , TL_{\sigma}\gamma(T) \right) \right).
\end{equation}
with
\begin{equation*}
\Delta(T) := |x_{0}-\overline{x}_{0}| + TD_{b}(X^*(T)) + M_B\, w(TK_{\sigma}(X^*(T)), TD_{\sigma}(X^*(T))).
\end{equation*}

Now, either $\gamma(T)\leq 1$ in which case Equation \eqref{eq:main:control:diffusion} is trivially satisfied, or $\gamma(T) >1$ in which case, using some property of the function $w$, Equation \eqref{eq:Gronwall:step} gives
\begin{equation*}
\gamma(T) \leq e^{L_bT} \Delta(T) + e^{L_bT}  M_B\, w\left( T K_{f}(X^*(T)+\overline{X}^*(T)) , TL_{f} \right) \sqrt{\gamma(T)}.
\end{equation*}
Yet, inequality of the form $\gamma\leq a+b\sqrt{\gamma}$ implies that $\gamma\leq 2a + b^{2}$ which ends the proof.

\end{proof}

\section{Applications of the stability inequality}
\label{sec:applications}

In this section, we use the notation $N>2$ for a scaling parameter and the notation $\alpha>0$ for a parameter which controls the rate of the scaling. Moreover, we consider $\overline{b}$ and $\overline{\sigma}$ two functions such that $\overline{b}$ is LB$(L_{b},K_{b})$ and $\overline{\sigma}^{2}$ is LB$(L_{\sigma},K_{\sigma})$. Finally, we assume for simplicity that the functions $K_{b}$ and $K_{\sigma}$ are constant in this whole section.

\subsection{Spatially independent diffusion coefficient}

Here, we assume that $L_\sigma=0$, that is $\overline{\sigma}(t,x)=\overline{\sigma}(t)$ is constant with respect to the space variable. For all $N>2$, let us define
\begin{equation*}
b^{N}(t,x) := \overline{b}(t,x) + N^{-\alpha}D_{b} \quad \text{ and } \sigma^{N}(t,x)^{2} := \overline{\sigma}(t,x)^{2} + N^{-2\alpha}D_{\sigma},
\end{equation*}
where $D_{b}$ and $D_{\sigma}$ are two constants for simplicity.
Let $\overline{x}_{0} \in \mathbb{R}$ and define $x_{0}^{N} = \overline{x}_{0} + N^{-\alpha}D_{x}$ for some constant $D_{x}\in \mathbb{R}$. Then, let $X^{N}$ and $\overline{X}$ satisfy
\begin{equation}\label{eq:EDS:XN:Xbar:Lsigma=0}
\begin{cases}
X^{N}(t) = x^{N}_{0} + \int_{0}^{t} b^{N}(s,X^{N}(s)) ds + B(\Lambda^{N}(t))\\
\overline{X}(t) = \overline{x}_{0} + \int_{0}^{t} \overline{b}(s,\overline{X}(s)) ds + B(\overline{\Lambda}(t)),
\end{cases}
\end{equation}
where 
\begin{equation*}
\Lambda^{N}(t) := \int_{0}^{t} \sigma^{N}(s)^{2}ds \quad \text{ and } \quad \overline{\Lambda}(t) := \int_{0}^{t} \overline{\sigma}(s)^{2}ds.
\end{equation*}

We are now in position to state the following corollary of Proposition \ref{prop:stability:inequality:diffusion}.

\begin{corollary}\label{cor:convergence:Lsigma:0}
    Let $\eta>0$. There exists a random variable $\Xi$ such that $\Xi^2$ admits exponential moments and, for all $N>1$,
    \begin{equation*}
        \sup_{0\leq t < \infty} \frac{|X^{N}(t) - \overline{X}(t)|}{e^{2(L_b + \eta)t}} \leq \Xi \frac{\ln N}{N^{\alpha}}.
    \end{equation*}
\end{corollary}

\begin{proof}
Let us consider the processes $Y^{N}$ and $\overline{Y}^{N}$ defined by $Y^{N} := N^{\alpha}X^{N}$ and $\overline{Y}^{N}:= N^{\alpha}\overline{X}^N$ which satisfy
\begin{equation}\label{eq:EDS:YN:Ybar:1}
\begin{cases}
Y^{N}(t) = N^{\alpha}x^{N}_{0} + \int_{0}^{t} N^{\alpha}b^{N}(s,N^{-\alpha}Y^{N}(s)) ds + \tilde{B}(\hat{\Lambda}^{N}(t))\\
\overline{Y}^{N}(t) = N^{\alpha}\overline{x}_{0} + \int_{0}^{t} N^{\alpha}\overline{b}(s,N^{-\alpha}\overline{Y}^{N}(s)) ds + \tilde{B}(\tilde{\Lambda}^{N}(t)),
\end{cases}
\end{equation}
where $\tilde{B}(t) = N^{\alpha}B(N^{-2\alpha} t)$ defines a standard Brownian motion and,
\begin{equation*}
\hat{\Lambda}^{N}(t) := \int_{0}^{t} N^{2\alpha}\sigma^{N}(s)^{2}ds \quad \text{ and } \quad \tilde{\Lambda}^{N}(t) := \int_{0}^{t} N^{2\alpha}\overline{\sigma}(s)^{2}ds.
\end{equation*}
The scaling used in the definition of $Y^{N}$ and $\overline{Y}^{N}$ magnifies the difference between $X^{N}$ and $\overline{X}^N$ (which is expected to be of order $N^{-\alpha}$). Hence, the difference between the $Y$ processes is expected to be of order $1$ and we are now in position to apply Proposition \ref{prop:stability:inequality:diffusion}. 

Let $M_{\tilde{B}}$ be the random variable of Theorem \ref{thm:expo:moments} associated with the Brownian motion $\tilde{B}$. The drift coefficients involved in Equation \eqref{eq:EDS:YN:Ybar:1} are $LBC(L_b, N^{\alpha}K_b + D_b, D_b)$ and the square diffusion coefficients are $LBC(0, N^{2\alpha}K_\sigma + D_\sigma, D_\sigma)$, so it follows that for all $T>0$, $|Y^{N}(T) - \overline{Y}^N(T)| \leq 1+ A(T)$, where 
\begin{equation*}
    A(T) = 2e^{2L_{b}T} \Big[ D_{x} + TD_{b} + M_{\tilde{B}}\, w(T(N^{2\alpha}K_{\sigma} + D_\sigma), TD_{\sigma}) \Big].
\end{equation*}
Yet, there exists some deterministic constant $C>0$ such that
\begin{equation*}
    D_{x} + TD_{b} \leq C e^{2\eta T} \quad \text{and} \quad  w(T(N^{2\alpha}K_{\sigma} + D_\sigma),TD_{\sigma}) \leq C e^{\eta T} \sqrt{1+ \ln N}.
\end{equation*}
Moreover, Corollary \ref{cor:scaled:BM} implies that $M_{\tilde{B}} \leq M_B (1 + \sqrt{2\alpha\varepsilon \ln N})$, where $M_{B}$ is the random variable of Theorem \ref{thm:expo:moments} associated with the initial Brownian motion $B$.

Finally, we have, for all $t>0$,
\begin{equation*}
\frac{|X^{N}(t) - \overline{X}^N(t)|}{e^{2(L_b+\eta)t}} \leq 2 C N^{-\alpha} \Big[ 1 + M_B (1 + \sqrt{2\alpha\varepsilon \ln N}) \sqrt{1+\ln N} \Big],
\end{equation*}
which gives the desired result since $M_{B}^2$ admits exponential moments.
\end{proof}

\subsection{Diffusion approximation of an ODE}

Here, $L_\sigma$ may be non null. For all $N>2$, let us define
\begin{equation*}
b^{N}(t,x) := \overline{b}(t,x) + N^{-\alpha}D_{b} \quad \text{ and } \sigma^{N}(t,x)^{2} := \overline{\sigma}(t,x)^{2} + N^{-\alpha}D_{\sigma},
\end{equation*}
where $D_{b}$ and $D_{\sigma}$ are two constants for simplicity.
Let $\overline{x}_{0} \in \mathbb{R}$ and define $x_{0}^{N} = \overline{x}_{0} + N^{-\alpha}D_{x}$ for some constant $D_{x}\in \mathbb{R}$. Then, let $X^{N}$ and $\overline{X}^N$ satisfy
\begin{equation}\label{eq:EDS:XN:Xbar}
\begin{cases}
X^{N}(t) = x^{N}_{0} + \int_{0}^{t} b^{N}(s,X^{N}(s)) ds + N^{-\alpha}B( N^{\alpha}\Lambda^{N}(t))\\
\overline{X}^N(t) = \overline{x}_{0} + \int_{0}^{t} \overline{b}(s,\overline{X}^N(s)) ds + N^{-\alpha}B(N^{\alpha}\overline{\Lambda}(t)),
\end{cases}
\end{equation}
where 
\begin{equation*}
\Lambda^{N}(t) := \int_{0}^{t} \sigma^{N}(s,X^{N}(s))^{2}ds \quad \text{ and } \quad \overline{\Lambda}^N(t) := \int_{0}^{t} \overline{\sigma}(s,\overline{X}^N(s))^{2}ds.
\end{equation*}

Notice that the diffusion part of Equation \eqref{eq:EDS:XN:Xbar} vanishes when $N$ goes to infinity. In particular, one could prove that both $X^N$ and $\overline{X}^N$ converge (at rate $N^{-\alpha/2}$) to the solution of the ordinary integral equation $\overline{x}(t) = \overline{x}_0 + \int_{0}^{t} \overline{b}(s,\overline{x}(s))ds$. The aim here is to prove that $X^N$ and $\overline{X}^N$ are close at the finer scale $N^{-\alpha}$ (up to logarithmic term).

We are now in position to state the following corollary of Proposition \ref{prop:stability:inequality:diffusion}.

\begin{corollary}\label{cor:convergence:ODE}
    Let $\eta>0$. There exists a random variable $\Xi$ with exponential moments such that, for all $N>1$,
    \begin{equation*}
        \sup_{0\leq t < \infty} \frac{|X^{N}(t) - \overline{X}^N(t)|}{e^{2(L_b + \eta)t}} \leq \Xi \frac{\ln N}{N^{\alpha}}.
    \end{equation*}
\end{corollary}

Notice that an equivalent result can be obtained from the proof of \cite[Lemma 3.2.]{kurtz1978strong}: for any $T>0$, 
\begin{equation*}
    \sup_{0\leq t < T} |X^{N}(t) - \overline{X}^N(t)| \leq \Xi Te^{2L_b T} \frac{\ln N}{N^{\alpha}}.
\end{equation*}
Hence, at the price of replacing a linear term in $t$ by an arbitrary small exponential term (and without any loss in the rate of convergence with respect to $N$) we are able to get a uniform control with respect to $t$.

\begin{proof}
Let us consider the processes $Y^{N}$ and $\overline{Y}^{N}$ defined by $Y^{N} := N^{\alpha}X^{N}$ and $\overline{Y}^{N}:= N^{\alpha}\overline{X}^N$ which satisfy
\begin{equation}\label{eq:EDS:YN:Ybar:2}
\begin{cases}
Y^{N}(t) = N^{\alpha}x^{N}_{0} + \int_{0}^{t} N^{\alpha}b^{N}(s,N^{-\alpha}Y^{N}(s)) ds + B(\hat{\Lambda}^{N}(t))\\
\overline{Y}^{N}(t) = N^{\alpha}\overline{x}_{0} + \int_{0}^{t} N^{\alpha}\overline{b}(s,N^{-\alpha}\overline{Y}^{N}(s)) ds + B(\tilde{\Lambda}^{N}(t)),
\end{cases}
\end{equation}
where 
\begin{equation*}
\hat{\Lambda}^{N}(t) := \int_{0}^{t} N^{\alpha}\sigma^{N}(s,N^{-\alpha}Y^{N}(s))^{2}ds \quad \text{ and } \quad \tilde{\Lambda}^{N}(t) := \int_{0}^{t} N^{\alpha}\overline{\sigma}(s,N^{-\alpha}\overline{Y}^{N}(s))^{2}ds.
\end{equation*}
The scaling used in the definition of $Y^{N}$ and $\overline{Y}^{N}$ magnifies the difference between $X^{N}$ and $\overline{X}^N$ (which is expected to be of order $N^{-\alpha}$). Hence, the difference between the $Y$ processes is expected to be of order $1$ and we are now in position to apply Proposition \ref{prop:stability:inequality:diffusion}. 

Let $M$ be the random variable of Theorem \ref{thm:expo:moments} associated with the Brownian motion $B$. The drift coefficients involved in Equation \eqref{eq:EDS:YN:Ybar:2} are $LBC(L_b, N^{\alpha}K_b + D_b, D_b)$ and the square diffusion coefficients are $LBC(L_\sigma, N^{\alpha}K_\sigma + D_\sigma, D_\sigma)$, so it follows that for all $T>0$, $|Y^{N}(T) - \overline{Y}^N(T)| \leq 1+ A(T)$, where 
\begin{equation*}
    A(T) = 2e^{2L_{b}T} \Big[ D_{x} + TD_{b} + M\, w(T(N^{\alpha}K_{\sigma} + D_\sigma), TD_{\sigma}) + M^{2}\, w\left( T(N^{\alpha}K_{\sigma} + D_\sigma), TL_{\sigma} \right)^{2} \Big]
\end{equation*}
Yet, there exists some deterministic constant $C>0$ such that
\begin{equation*}
    D_{x} + TD_{b} \leq C e^{2\eta T} \quad \text{and} \quad  w(T(N^{\alpha}K_{\sigma} + D_\sigma), \max\{ TD_{\sigma}, TL_{\sigma} \}) \leq C e^{\eta T} \sqrt{1+ \ln N}.
\end{equation*}
Finally, we have, for all $t>0$,
\begin{equation*}
\frac{|X^{N}(t) - \overline{X}^N(t)|}{e^{2(L_b+\eta)t}} \leq 2 C N^{-\alpha} \Big[ 1 + M\, \sqrt{1+\ln N} + M^{2}\, (1+\ln N) \Big],
\end{equation*}
which gives the desired result since $M^2$ admits exponential moments.
\end{proof}

\section{Proof of Theorem \ref{thm:expo:moments}}
\label{sec:proof:thm}
\subsection{A modified Garsia–Rodemich–Rumsey lemma}

Let us introduce some notation. Let $\Psi$ and $\mu$ be two non decreasing functions from $\mathbb{R}_{+}$ to $\mathbb{R}_{+}$. Furthermore, assume that $\mu$ is continuous, $\mu(0)=0$, $\lim_{x\to +\infty} \Psi(x) = +\infty$ and define $\Psi^{-1} : [\Psi(0),+\infty)$ by
\begin{equation*}
\Psi^{-1}(u) := \sup \{v,\, \Psi(v)\leq u\}.
\end{equation*}

The following lemma is a simple extension of \cite[Lemma 1]{garsia1970real}.
\begin{lemma}\label{lem:modified:GRR}
For any $T>0$, let $f:[0,T]\to \mathbb{R}$ be a continuous function such that
\begin{equation*}
\int_{0}^{T} \int_{0}^{T} \Psi\left(\frac{|f(t)-f(s)|}{\mu(t-s)}\right) dtds \leq B_{T} <+\infty.
\end{equation*}
Then, for all $t,s\in [0,T]$, 
\begin{equation}\label{eq:modified:GRR}
|f(t)-f(s)| \leq 8 \int_{0}^{|t-s|} \Psi^{-1}\left(\frac{4B_{T}}{u^{2}}\right)  d\mu(u).
\end{equation}
\end{lemma}
\begin{proof}
For any $T>0$, let $f_{T}:[0,1]\to \mathbb{R}$ be defined by $\tilde{f}(x) := f(Tx)$ and let us define $\mu_{T}$ in the same way. By a change of variable $Tx\to t$ and $Ty\to s$, we have
\begin{equation*}
\int_{0}^{1} \int_{0}^{1} \Psi\left(\frac{|f_{T}(x)-f_{T}(y)|}{\mu_{T}(x-y)}\right) dxdy = \frac{1}{T^{2}} \int_{0}^{T} \int_{0}^{T} \Psi\left(\frac{|f(t)-f(s)|}{\mu(t-s)}\right) dtds \leq \frac{B_{T}}{T^{2}}.
\end{equation*}
Then, the functions $f_{T}$, $\Psi$ and $\mu_{T}$ satisfies the assumption of \cite[Lemma 1]{garsia1970real} which implies that, for all $x,y\in [0,1]$,
\begin{equation*}
|f_{T}(x)-f_{T}(y)| \leq 8 \int_{0}^{|x-y|} \Psi^{-1}\left(\frac{4B_{T}}{T^{2}v^{2}}\right) d\mu_{T}(v).
\end{equation*}
Finally, the change of variable $Tx\to t$, $Ty\to s$ and $Tv\to u$ gives \eqref{eq:modified:GRR}
\end{proof}

The rest of the proof relies on an application of Lemma \ref{lem:modified:GRR} with the functions $\Psi$ and $\mu$ defined by, for all $x\in \mathbb{R}_{+}$,
\begin{equation*}
\Psi(x) := e^{x^{2}/2} - 1 \quad \text{ and } \quad \mu(x):= \sqrt{cx},
\end{equation*}
where $c>1$ is some constant. Notice that $\Psi^{-1}(y) = \sqrt{2\ln(y+1)}$ and $d\mu(x) = \frac{\sqrt{c}}{2\sqrt{x}}dx$.

Let $\varepsilon>0$. For all real number $T> 0$, let us define the random variable
\begin{equation*}
    \xi_{T} := f_{\varepsilon}(T) \int_{0}^{T} \int_{0}^{T} \Psi\left(\frac{|B_{t}-B_{s}|}{\mu(|t-s|)}\right) dsdt,
\end{equation*}
where
\begin{equation}
    f_{\varepsilon}(T) := 
    \begin{cases}
        (1/T+1)^{2(1-\varepsilon)} & \text{if $T<1$},\\
        1 & \text{if $T=1$},\\
        (T-1)^{-2(1+\varepsilon)} & \text{if $T>1$}.
    \end{cases}
\end{equation}
In particular, for any positive number $t$, we have: 1) if $t\leq 1$, $\frac{1}{\lfloor 1/t \rfloor +1} < t \leq \frac{1}{\lfloor 1/t \rfloor}$ and $f_{\varepsilon}(\frac{1}{\lfloor 1/t \rfloor})\geq t^{-2(1-\varepsilon)}$; 2) if $t\geq 1$, $\lceil t \rceil - 1 < t \leq \lceil t \rceil$ and $f_{\varepsilon}(\lceil t \rceil)\geq t^{-2(1+\varepsilon)}$.
Finally, let us denote by $\xi$ the sup over all integer or inverse integer times, that is $\xi := \sup\{ \xi_{T}, \xi_{1/T} \,|\, T\in \mathbb{N}^{*}\}$.

\begin{lemma}\label{lem:finite:moment:xi}
For all $p\in (1,c)$, $\mathbb{E}[\xi^{p}]< +\infty$.
\end{lemma}
\begin{proof}
Let $p\in (1,c)$ and $q\in (1,p)$. For all positive integers $T\geq 1$, we have by convexity,
\begin{eqnarray}
\mathbb{E}\left[\xi_{T}^{q}\right] \leq \mathbb{E}\left[(\xi_{T}+T^{2}f_{\varepsilon}(T))^{q}\right] & = & f_{\varepsilon}(T)^{q} \, \mathbb{E}\left[\left(\int_{0}^{T} \int_{0}^{T} \exp\left(\frac{|B_{t}-B_{s}|^{2}}{2c|t-s|}\right) dsdt\right)^{q}\right]\\
& \leq & f_{\varepsilon}(T)^{q}\,  T^{2(q-1)} \mathbb{E}\left[\int_{0}^{T} \int_{0}^{T} \exp\left(\frac{|B_{t}-B_{s}|^{2}}{2c|t-s|}\right)^{q} dsdt\right]\\
 & =& \frac{T^{2q}}{(T-1)^{(2+\varepsilon)q}} \, T^{-2} \int_{0}^{T} \int_{0}^{T} \mathbb{E}\left[\exp\left(\frac{q}{2c}\left(\frac{|B_{t}-B_{s}|}{\sqrt{|t-s|}}\right)^{2}\right)\right] dsdt.
\end{eqnarray}
Yet, since the increments of $B$ are gaussian, for all $t\neq s$,
\begin{equation*}
\mathbb{E}\left[\exp\left(\frac{q}{2c}\left(\frac{|B_{t}-B_{s}|}{\sqrt{|t-s|}}\right)^{2}\right)\right] = \frac{\sqrt{c}}{\sqrt{c-q}}.
\end{equation*}
Hence, 
\begin{equation*}
\mathbb{E}\left[\xi_{T}^{q}\right]\leq \frac{T^{2q}}{(T-1)^{(2+\varepsilon)q}} \frac{\sqrt{c}}{\sqrt{c-q}},
\end{equation*}
and similarly, $\mathbb{E}\left[\xi_{1/T}^{q}\right]\leq \frac{(T+1)^{(2-\varepsilon)q}}{T^{2q}} \frac{\sqrt{c}}{\sqrt{c-q}}$. Denote $g(T) := \frac{T^{2q}}{(T-1)^{(2+\varepsilon)q}} + \frac{(T+1)^{(2-\varepsilon)q}}{T^{2q}}$ and remark that $g(T)$ is equivalent to $2T^{-\varepsilon q}$ as $T\to \infty$ which in turn implies summability since $\varepsilon>0$.
By Markov's inequality and the union bound, for all integer $n\geq 0$, $\mathbb{P}\left(\max(\xi_{T}^{p}, \xi_{1/T}^{p}) > n\right) = \mathbb{P}\left(\max(\xi_{T}^{q}, \xi_{1/T}^{q}) > n^{q/p}\right) \leq g(T)n^{-q/p}$. Then, the union bound gives
\begin{equation*}
\mathbb{P}\left(\xi^{p}>n\right) \leq \sum_{T=1}^{+\infty} g(T) n^{-q/p} \leq C n^{-q/p}.
\end{equation*}
Finally, the fact that $q/p<1$ gives the result (use for instance the fact that $\mathbb{E}\left[ \xi^{p} \right] \leq \sum_{n=0}^{+\infty} \mathbb{P}\left(\xi^{p}>n\right)$).
\end{proof}

We are now in position to prove Theorem \ref{thm:expo:moments}. Let us first fix some $t \geq 1$. By definition of $\xi$ and properties of the function $f_{\varepsilon}$, we have
\begin{equation*}
\int_{0}^{t} \int_{0}^{t} \Psi\left(\frac{|B_{x}-B_{y}|}{\mu(|x-y|)}\right) dxdy \leq f_{\varepsilon}(\lceil t \rceil)^{-1} \xi_{\lceil t \rceil} \leq t^{2(1+\varepsilon)} \xi,
\end{equation*}
Hence, Lemma \ref{lem:modified:GRR} implies that
\begin{equation*}
\forall x,y\in [0,t], \quad |B_{x}-B_{y}| \leq 8 \int_{0}^{|x-y|} \Psi^{-1}\left(\frac{4t^{2(1+\varepsilon)} \xi}{u^{2}}\right) d\mu(u).
\end{equation*}
Specializing the equation above with $x=t$ and $y=s\leq t$ and denoting $h=|t-s|$ yields
\begin{equation*}
|B_{t}-B_{s}|\leq 8 \int_{0}^{h} \sqrt{2\ln\left(\frac{4t^{2(1+\varepsilon)} \xi}{u^{2}}+1\right)} \frac{\sqrt{c}}{2\sqrt{u}} du,
\end{equation*}
and so
\begin{equation*}
|B_{t}-B_{s}|  \leq  4\sqrt{2c} \int_{0}^{h} \sqrt{\ln\left(4\xi + \frac{u^{2}}{t^{2(1+\varepsilon)}}\right) + \ln\left(\frac{t^{2(1+\varepsilon)}}{u^{2}}\right)} \frac{du}{\sqrt{u}}.
\end{equation*}
The ratio $u^{2}/t^{2(1+\varepsilon)}$ is less than 1 so the second logarithm in the equation above is positive and we can use the inequality $\sqrt{a+b}\leq \sqrt{a}+\sqrt{b}$ to get, for all $s\leq t$ and $t\geq 1$, $|B_{t}-B_{s}| \leq I_{1}+I_{2}+I_3$ with 
\begin{equation*}
\begin{cases}
I_{1} := 4\sqrt{2c} \sqrt{\ln\left(4\xi + 1\right)} \int_{0}^{h} \frac{du}{\sqrt{u}},\\
I_{2} := 8\sqrt{c} \int_{0}^{h}  \sqrt{\ln \frac{t}{u} + \varepsilon|\ln t|} - \frac{1}{\sqrt{\ln \frac{t}{u} + \varepsilon|\ln t|}}  \frac{du}{\sqrt{u}},\\
I_3 :=  8\sqrt{c} \int_{0}^{h} \frac{1}{\sqrt{\ln \frac{t}{u} + \varepsilon|\ln t|}} \frac{du}{\sqrt{u}}.
\end{cases}
\end{equation*}
If $t\leq 1$, one can use the fact that
\begin{equation*}
    \int_{0}^{t} \int_{0}^{t} \Psi\left(\frac{|B_{x}-B_{y}|}{\mu(|x-y|)}\right) dxdy \leq \left( f_{\varepsilon}(1/\lfloor 1/t \rfloor) \right)^{-1} \xi_{1/\lfloor 1/t \rfloor} \leq t^{2(1-\varepsilon)} \xi,
\end{equation*}
and the same arguments as above to prove that the bound $|B_{t}-B_{s}| \leq I_{1}+I_{2}+I_3$ is also valid for all $s\leq t$ and $t\leq 1$.

First, $I_{1}\leq 8\sqrt{2c} \sqrt{\ln\left(4\xi + 1\right)} \sqrt{h}$. Then, to simplify the expressions of $I_2$ and $I_3$, let us denote $a$ such that $\ln a = \ln t + \varepsilon |\ln t|$, so that $\ln \frac{t}{u} + \varepsilon|\ln t| = \ln (a/u)$. Remark that $a\geq t$. The integrand in $I_{2}$ is the derivative of $u\mapsto 2\sqrt{u \ln (a/u)}$, so that $I_2 = 16\sqrt{c}\sqrt{h \ln (a/h)}$. Finally, with the change of variable $y=\sqrt{\ln (a/u)}/\sqrt{2}$, we have
\begin{equation*}
    I_3 = 8\sqrt{c} \sqrt{2\pi}\sqrt{a} \left( 1 - \operatorname{erf}\left( \frac{\sqrt{\ln (a/h)}}{\sqrt{2}} \right) \right),
\end{equation*}
where $\operatorname{erf}$ is the error function defined by $\operatorname{erf}(x) = 2/\sqrt{\pi} \int_{0}^{x} e^{-y^{2}} dy$. If $h\leq a/2$, we use the classic bound $1-\operatorname{erf}(x) \leq e^{-x^{2}}/(x\sqrt{\pi})$ to get
\begin{equation*}
    I_3 \leq 8\sqrt{c} \sqrt{2\pi} \sqrt{a}  \frac{\sqrt{h/a}}{\sqrt{\pi} \frac{\sqrt{\ln (a/h)}}{\sqrt{2}}}
        \leq \frac{16\sqrt{c}}{\sqrt{\ln 2}} \sqrt{h}.
\end{equation*}
If $h\geq a/2$, we use $1-\operatorname{erf}(x) \leq 1$ and $\sqrt{a}\leq \sqrt{2h}$ to get $I_3 \leq 16\sqrt{c}\sqrt{\pi} \sqrt{h}$. Since $(\ln 2)^{1/2} \leq \sqrt{\pi}$, we have for all $h\leq t$, $I_3 \leq 16\sqrt{c}\sqrt{\pi} \sqrt{h}$.

Remind that $\ln (a/u) = \ln \frac{t}{u} + \varepsilon|\ln t|$ and combine the bounds on $I_1$, $I_2$ and $I_3$ to get, for all $t\geq 0$ and $s\leq t$, with $h=|t-s|$,
\begin{equation*}
    |B_{t}-B_{s}| \leq 8\sqrt{2c} \left( \sqrt{\ln\left(4\xi + 1\right)} + \sqrt{2} (1+\sqrt{\pi}) \right) \sqrt{h \left( 1+\ln \frac{t}{h} + \varepsilon|\ln t| \right)}.
\end{equation*}
This inequality holds for all $0\leq s<t<+\infty$ which implies that the random variable $M$ defined in the statement of the Theorem satisfies
\begin{equation*}
M \leq 8\sqrt{2c} \left( \sqrt{\ln\left(4\xi + 1\right)} + \sqrt{2} (1+\sqrt{\pi}) \right),
\end{equation*}
and so, using $1+\sqrt{\pi} \leq 4$, we have $M^{2} \leq 256c\ln(4\xi + 1) + 8192$. In particular, it implies that $M$ is almost surely finite. Moreover, for $\lambda>0$, 
\begin{equation*}
\mathbb{E}\left[e^{\lambda M^{2}}\right]\leq e^{8192\lambda} \mathbb{E}\left[\exp(256c\ln(4\xi + 1))\right]\leq e^{8192\lambda} \mathbb{E}\left[(4\xi+1)^{256c\lambda}\right].
\end{equation*}
Finally, $\lambda$ can be chosen such that $1 < 256c\lambda < c$ and Lemma \ref{lem:finite:moment:xi} gives the result.

\paragraph*{Acknowledgments}

This research has been supported by ANR-19-CE40-0024 (CHAllenges in MAthematical NEuroscience) and has been conducted while the author was in Statify team at Centre Inria de l'Université Grenoble Alpes. The author would also like to thank Markus Fischer for fruitful discussions on the subject.

\bibliographystyle{abbrv}
\bibliography{references.bib}

\end{document}